\documentclass{svjour3}                     
\smartqed  

\usepackage{graphicx}

\usepackage{amssymb}
\usepackage{amsfonts}
\usepackage{latexsym}
\usepackage{amscd}

\usepackage{ae,enumerate}

\usepackage[mathscr]{euscript}
\usepackage{xy} \xyoption{all}
\newtheorem{thm}{Theorem}
\newtheorem{exam}{Example}
\newtheorem{defn}{Definition}

\newtheorem{prop}{Proposition}

\newtheorem{cor}{Corollary}

\newtheorem{rem}{Remark}

\newcommand{\uloopr}[1]{\ar@'{@+{[0,0]+(-4,5)}@+{[0,0]+(0,10)}@+{[0,0] +(4,5)}}^{#1}}
\newcommand{\uloopd}[1]{\ar@'{@+{[0,0]+(5,4)}@+{[0,0]+(10,0)}@+{[0,0]+ (5,-4)}}^{#1}}
\newcommand{\dloopr}[1]{\ar@'{@+{[0,0]+(-4,-5)}@+{[0,0]+(0,-10)}@+{[0, 0]+(4,-5)}}_{#1}}
\newcommand{\dloopd}[1]{\ar@'{@+{[0,0]+(-5,4)}@+{[0,0]+(-10,0)}@+{[0,0 ]+(-5,-4)}}_{#1}}

\newcommand{\luloop}[1]{\ar@'{@+{[0,0]+(-8,2)}@+{[0,0]+(-10,10)}@+{[0, 0]+(2,2)}}^{#1}}

\newcommand{\Z}{\mathbb{Z}}

\newcommand{\ra}{\rightarrow}

\newcommand{\ideal}{\lhd}

\newcommand{\setmid}[2]{\{\,  #1 \mid #2 \,\} }

\newcommand{\extcent}[1]{C(#1)}

\newcommand{\mart}[1]{{\rm Frac}(#1)}

\begin{document}

\title{The Dixmier-Moeglin equivalence for Leavitt path algebras
 \thanks{The first author is partially supported by the U.S. National Security Agency under grant number H89230-09-1-0066. The second author was supported by NSERC grant 31-611456.}
}


\author{Gene Abrams, \  Jason P. Bell, \ and    Kulumani M. Rangaswamy \\ 
}


\institute{G. Abrams \at
              University of Colorado at Colorado Springs \\
              Colorado Springs, CO 80918 \ U.S.A. \\
              Tel.: 7192553182\\
              Fax: 7192553605\\
              \email{abrams@math.uccs.edu}           
           \and
           J. Bell  \at
              Simon Fraser University \\
              Burnaby, BC, Canada \ V5A 1S6 \\
              \email{jpb@math.sfu.ca}
              \and
              K. M. Rangaswamy  \at
              University of Colorado at Colorado Springs \\
              Colorado Springs, CO 80918 \ U.S.A. \\
              \email{krangasw@uccs.edu}
}

\date{Received: date / Accepted: date}

\maketitle

\begin{abstract}
Let $K$ be a field, let $E$ be a finite directed graph, and let $L_K(E)$ be the Leavitt path algebra of $E$ over $K$.  We show that for a prime ideal $P$ in $L_K(E)$, the following are equivalent:
\begin{enumerate}
\item $P$ is primitive;
\item $P$ is rational;
\item $P$ is locally closed in ${\rm Spec}(L_K(E))$.
\end{enumerate}
We show that the prime spectrum ${\rm Spec}(L_K(E))$ decomposes  into a finite disjoint union of subsets, each of which is homeomorphic to ${\rm Spec}(K)$ or to ${\rm Spec}(K[x,x^{-1}])$.  In the case that $K$ is infinite, we show that $L_K(E)$ has a rational $K^{\times}$-action, and that the indicated decomposition of ${\rm Spec}(L_K(E))$ is induced by this action.
\keywords{Leavitt path algebras \and  Dixmier-Moeglin equivalence \and  stratification \and affine algebraic group actions \and  prime spectrum \and  primitive spectrum.}
 \subclass{16G20 \and 14L30 \and 05E10}
\end{abstract}


\section{Introduction}\label{Introduction}
One of the important questions in the study of noncommutative algebras is to determine the structure of simple right modules.  In general, this is a very difficult problem.  Dixmier's approach to the classification problem for simple right modules of an algebra $A$ is to first determine the primitive ideals of $A$.  For each primitive ideal $P$ of $A$, one then attempts to classify the simple right $A$-modules whose annihilator is equal to $P$.  Consequently, understanding the primitive ideals of an algebra is of great importance and there is a long history of the study of primitive ideals and primitive rings.

Dixmier and Moeglin \cite{Dixmier}, \cite{Moeglin} gave a concrete characterization of the primitive ideals in an enveloping algebra $U(\mathcal{L})$ of a finite-dimensional complex Lie algebra $\mathcal{L}$.  They showed that for a prime ideal of $U(\mathcal{L})$, primitivity is equivalent to two other easy-to-recognize properties: being \emph{locally closed} and being \emph{rational}.

Given an algebra $A$, we say that a prime ideal $P\in {\rm Spec}(A)$ is \emph{locally closed} if it is an open subset of the closure of $\{P\}$ in the Zariski topology.  Equivalently, $P$ is locally closed if the intersection of all prime ideals in $A$ that properly contain $P$ strictly contains $P$.  In the case when ${\rm Spec}(A)$ satisfies the descending chain condition, this is equivalent to saying that $A/P$ has finitely many height one primes.

 When $A$ is noetherian and $P$ is a prime ideal of $A$, the ring $A/P$ is again noetherian and hence we can invert the nonzero regular elements of $A$ to form the Goldie ring of quotients, which we denote by $Q(A/P)$.  This is a simple Artinian ring and its center is a field.  We say that $P$ is \emph{rational} if the center $Z(Q(A/P))$ of $Q(A/P)$ is an algebraic extension of the base field of $A$.  On the other hand, when we are dealing with a non-noetherian algebra, we cannot guarantee that a prime homomorphic image will have a Goldie ring of quotients.  Several authors have noted that one can get around this obstruction by instead working with the Martindale ring of quotients, which we define in \S \ref{Mart}.  The Martindale ring of quotients of a prime algebra $A$ is a simple ring and its center is a field, which we call the \emph{extended centroid} of $A$.  Thus we can extend the notion of rational prime ideals to non-noetherian rings, by declaring that a prime ideal $P$ of an algebra $A$ is \emph{rational} if the extended centroid of $A/P$ is an algebraic extension of $K$.  This approach was used by Lorenz in extending the Dixmier-Moeglin equivalence to non-noetherian rings supporting a rational action of an affine algebraic group \cite{Lorenz1}, \cite{Lorenz2}.

\begin{defn} {\em Let $K$ be a field and let $A$ be a $K$-algebra.  We say that $A$ satisfies the \emph{Dixmier-Moeglin equivalence} on prime ideals if for $P\in {\rm Spec}(A)$ the following conditions are equivalent:
\begin{enumerate}
\item $P$ is primitive;
\item $P$ is rational;
\item $P$ is locally closed.
\end{enumerate}}
\end{defn}

Since the work of Dixmier and Moeglin on enveloping algebras, the Dixmier-Moeglin equivalence has been shown to hold for many classes of algebras, including algebras satsfying a polynomial identity, various quantum groups \cite{GL}, \cite{BG}, and many of the algebras coming from noncommutative projective geometry \cite{BRS}. The prime spectrum of algebras which satisfy the Dixmier-Moeglin equivalence are generally well behaved and often have a nice topological structure.  For example, in the case of many quantum groups, Goodearl and Letzter \cite{GL} have shown that the prime spectrum decomposes naturally into finitely many strata (which are constructed from a rational torus action), with each stratum homeomorphic to the scheme of irreducible subvarieties of a torus.  Lorenz  \cite{Lorenz1}, \cite{Lorenz2} has shown that this stratification property is a general phenomenon that one can attach to algebras supporting a rational affine algebraic group action and that the Dixmier-Moeglin equivalence is closely related to the number of strata being finite.

Given a row-finite directed graph $E$ and a field $K$, one can construct the \emph{Leavitt path algebra}, $L_K(E)$, which is an associative algebra over $K$.  (In the case that $E$ has an infinite set of vertices, $L_K(E)$ does not have a multiplicative identity.).  The original constructions due to Leavitt correspond to a very specific collection of graphs, but it has since been shown that a similar construction can be done for arbitrary directed graphs.  Leavitt path algebras are an algebraic analogue of a certain class of $C^*$-algebras (which properly belong to analysis), and they have received much recent attention.   We define Leavitt path algebras and give the necessary background for their study in \S \ref{LPA}.  Our main result is the following theorem.
\begin{thm} Let $E$ be a finite directed graph and let $K$ be a field.  Then the Dixmier-Moeglin equivalence holds for prime ideals of $L_K(E)$.
\label{thm: main}
\end{thm}
A major tool in our study is the work of Aranda Pino, Pardo, and Siles Molina \cite{PPSM}, who studied the prime and primitive spectrum of $L_K(E)$.  We show that their results can be recast in the language of stratification.  In particular, if $K$ is infinite, we show that a Leavitt path algebra $L_K(E)$ supports a rational $K^{\times}$-action as $K$-algebra automorphisms.  In the case that $E$ is finite, this action partitions the prime spectrum of $L_K(E)$ into finitely many strata with each stratum homeomorphic to either ${\rm Spec}(K)$ or ${\rm Spec}(K[x,x^{-1}])$; furthermore, the primitive ideals are precisely those ideals that are maximal in their stratum (see Theorem \ref{strat} and the discussion afterwards for a precise statement).    We use this stratification result to prove that the Dixmier-Moeglin equivalence holds (over any base field) for the Leavitt path algebra of any finite directed graph.

The outline of this paper is as follows.  In \S \ref{LPA}, we give background on Leavitt path algebras and relevant terminology.  In \S \ref{Mart}, we recall the construction of the Martindale ring of quotients of a prime algebra.   In \S \ref{Stratification}, we describe Goodearl and Letzter's partition of the prime spectrum of algebras supporting a rational affine algebraic group action and show that these results apply to Leavitt path algebras.  In \S \ref{DM}, we prove Theorem \ref{thm: main}.


\section{Leavitt path algebras}
\label{LPA}
In this section, we quickly recall the construction of a Leavitt path algebra.
A more complete description can be found in
\cite{AA1}. A \emph{directed graph}
$E=(E^0,E^1,r_E,s_E)$ consists of two sets $E^0,E^1$ and maps $r_E,s_E:E^1
\to E^0$.  The elements of $E^0$ are called \emph{vertices} and the
elements of $E^1$ \emph{edges}. We write $s$ for $s_E$ (resp., $r$ for $r_E$) if the graph $E$ is clear from context.  We emphasize that loops and
parallel edges are allowed.

If $s^{-1}(v)$ is a finite set for
every $v\in E^0$, then the graph is called \emph{row-finite}.  If both $E^0$ and $E^1$ are finite sets we call $E$ \emph{finite}. (All
graphs in this paper will be assumed to be row-finite. Although our main result will apply only to finite graphs, we present a number of the lead-up results in this more general setting.)  A vertex
$v$ for which $s^{-1}(v)$ is empty is called a \emph{sink};  a
vertex $w$ for which $r^{-1}(w)$ is empty is called a \emph{source}.

A \emph{path} $\mu$ in a graph $E$ is a sequence of edges
$\mu=e_1\dots e_n$ such that $r(e_i)=s(e_{i+1})$ for
$i=1,\dots,n-1$. In this case, $s(\mu):=s(e_1)$ is the \emph{source}
of $\mu$, $r(\mu):=r(e_n)$ is the \emph{range} of $\mu$, and $n$ is
the \emph{length} of $\mu$.  An edge $f$ is an {\it exit} for a path
$\mu = e_1 \dots e_n$ if there exists $i$ such that $s(f)=s(e_i)$
and $f \neq e_i$. If $\mu$ is a path in $E$, and if
$v=s(\mu)=r(\mu)$, then $\mu$ is called a \emph{closed path based at
$v$}. If $\mu= e_1 \dots e_n$ is a closed path based at $v = s(\mu)$
and $s(e_i)\neq s(e_j)$ for every $i\neq j$, then $\mu$ is called a
\emph{cycle}.  We define a relation $\ge$ on $E^0$ by declaring that $v\ge w$ in case $v=w$, or in case there is a path $\mu$ with $s(\mu) =v$ and $r(\mu)=w$.    A subset $H\subseteq E^0$ is called \emph{hereditary} if
$$v\in H \ {\rm and} \ v\geq w\Rightarrow w\in H.$$
A hereditary set $H$ is called \emph{saturated} if
$$s^{-1}(v)\neq \emptyset~{\rm and}~r(s^{-1}(v))\subseteq H\Rightarrow v\in H.$$
In words, $H$ is saturated in case whenever $v$ is a vertex having the property that all of the vertices to which the edges emanating from $v$ point are in $H$, then $v$ is in $H$ as well.

We say the graph $E$ {\it satisfies Condition (L)} in case every cycle in $E$ has an exit.

Our focus in this article is on $L_K(E)$, the Leavitt path algebra
of $E$. We define $L_K(E)$ and give a
few examples.

\begin{defn}\label{definition}  {\rm Let $E$ be any row-finite graph, and $K$ any field.
The {\em Leavitt path $K$-algebra} $L_K(E)$ {\em of $E$ with coefficients in $K$} is
the $K$-algebra generated by a set $\{v \mid v\in E^0\}$ of pairwise orthogonal idempotents,
together with a set of variables $\{e,e^* \mid e \in E^1 \}$, which satisfy the following
relations:

(1) $s(e)e=er(e)=e$ for all $e\in E^1$.

(2) $r(e)e^*=e^*s(e)=e^*$ for all $e\in E^1$.

(3) (The ``CK1 relations") \ $e^*e'=\delta _{e,e'}r(e)$ for all $e,e'\in E^1$.

(4) (The ``CK2 relations") \ $v=\sum _{\{ e\in E^1\mid s(e)=v \}}ee^*$ for every vertex
$v\in E^0$ for which $s^{-1}(v)$ is
nonempty.
}
\end{defn}

We will sometimes denote $L_K(E)$ simply by $L(E)$ for notational convenience.  The set
$\{e^*\mid e\in E^1\}$ will be denoted by $(E^1)^*$. We let $r(e^*)$
denote $s(e)$, and we let $s(e^*)$ denote $r(e)$. If $\mu = e_1
\cdots e_n$ is a path, then we denote by $\mu^*$ the element $e_n^*
\cdots e_1^*$ of $L_K(E)$.

An alternate description of $L_K(E)$ is given by the first author and Aranda Pino \cite{AA1}, where
it is described in terms of a free associative algebra modulo the
appropriate relations indicated in Definition \ref{definition}
above. As a consequence, if $A$ is any $K$-algebra which contains a
set of elements satisfying these same relations (we call such a set
an $E$-\emph{family}), then there is a (unique) $K$-algebra
homomorphism from $L_K(E)$ to $A$ mapping the generators of $L_K(E)$
to their appropriate counterparts in $A$.

Many well-known algebras arise as the Leavitt path algebra of a
row-finite graph \cite[Examples 1.4]{AA1}.
For example, the classical Leavitt $K$-algebras $L_n$ for $n\ge 2$ arise as the
algebras $L_K(R_n)$ where $R_n$ is the ``rose with $n$ petals" graph (see Figure $1$).
\begin{figure}
$$\xymatrix{ & {\bullet^v} \ar@(ur,dr) ^{e_1} \ar@(u,r) ^{e_2}
\ar@(ul,ur) ^{e_3} \ar@{.} @(l,u) \ar@{.} @(dr,dl) \ar@(r,d) ^{e_n}
\ar@{}[l] ^{\ldots} }$$
\caption{The rose with $n$ petals.}
\end{figure}
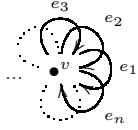 The full $n\times n$ matrix algebra over $K$
arises as the Leavitt path algebra of the oriented $n$-line graph (see Figure $2$).
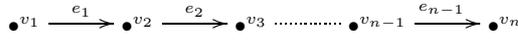
\begin{figure}
$$\xymatrix{{\bullet}^{v_1} \ar [r] ^{e_1} & {\bullet}^{v_2} \ar [r]
^{e_2} & {\bullet}^{v_3} \ar@{.}[r] & {\bullet}^{v_{n-1}} \ar [r]
^{e_{n-1}} & {\bullet}^{v_n}} $$

\caption{The oriented $n$-line graph.}
\end{figure}
The Laurent polynomial
algebra $K[x,x^{-1}]$ arises as the Leavitt path algebra of the
``one vertex, one loop" graph (see Figure $3$).
\begin{figure}
$$\xymatrix{{\bullet}^{v} \ar@(ur,dr) ^x}$$
\caption{The ``one vertex, one loop'' graph.}
\end{figure}
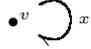

 Given a directed graph $E=(E^0,E^1)$ and a hereditary subset $H$ of $E^0$ we create the \emph{quotient graph}, denoted $E/H$, whose vertex set is given by $(E/H)^0:=E^0\setminus H$ and whose
edge set, $(E/H)^1$, consists of all edges $e\in E^1$ such that $r(e)\not\in H$.  (Note that since $H$ is hereditary, if $e\in (E/H)^1$ then $s(e)\not\in H$.)
 It is easy to show that if $K$ is a field, $E$ is a row-finite directed graph, and $I$ is an ideal of $L_K(E)$, then $E^0\cap I$ is a hereditary saturated subset of $E^0$.   Conversely, if $H$ is a hereditary saturated subset of $E^0$ and $I = I(H)$ is the ideal of $L_K(E)$ generated by $H$, then $I\cap H = H$.  Indeed, in this situation we have $L_K(E)/I(H) \cong L_K(E/H)$.

A property of Leavitt path algebras, of which we will make use on more than one occasion, bears mentioning here.

\begin{prop}\label{structureofvL(E)v} (\cite[Lemma 1.5]{ABGM})   Let $E$ be an arbitrary graph, and $c$ a cycle based at $v$ for which $c$ has no exits in $E$.   Then $vL_K(E)v \cong K[x,x^{-1}]$ as $K$-algebras, via an isomorphism $\varphi$ for which $\varphi(c) = x$.
\end{prop}

\section{The Martindale ring of quotients}
\label{Mart}
In this section, we recall Martindale's construction of the ring of quotients of a prime ring.  We include the full definition of Martindale's ring of quotients, but we refer to other sources \cite{Lam}, \cite{Rowen}, \cite{BMM}  for the proof that this construction actually yields a ring with the stated properties.  We note that this construction does not require that the ring have a multiplicative identity.

Let $A$ be a prime ring and consider the set of all right $A$-module homomorphisms $f: I \ra A$, where $I$ ranges over all nonzero two-sided ideals of $A$.  Martindale \cite{Martindale1} shows that one can endow this set of maps with a useful algebraic structure.

\begin{defn}\label{def:martindale-ring-of-quotients}{\em
Let $K$ be a field and let $A$ be a prime $K$-algebra.  The \emph{(right) Martindale ring of quotients of $A$}, denoted $\mart{A}$, consists of equivalence classes of pairs $(I, f)$ where $I \ideal A$, $I\neq (0)$, and $f\in {\rm Hom}_{A}(I_{A}, A_{A})$.  Here two pairs $(I,f)$, $(J, g)$ are defined to be equivalent if $f = g$ on the intersection $I\cap J$.  Addition and multiplication are given by
$$(I, f) + (J, g) \ = \ (I\cap J, f+g), \ \ \ \ \
(I,f)\cdot(J,g) \ = \ (JI, f\circ g).$$
}
\end{defn}

\begin{defn}\label{extendedcentroiddef}
Let $K$ be a field and let $A$ be a prime $K$-algebra.  The \emph{extended centroid of $A$}, written $\extcent{A}$, is defined to be $Z(\mart{A})$.
\end{defn}
We note that if $A$ is a prime algebra then $A$ embeds in $\mart{A}$ via the map
$a\mapsto [(A,f_a)]$, where $[(A,f_a)]$ denotes the equivalence class of the map $f_a:A\ra A$  defined by $f_a(x)=ax$.

\begin{rem}
If $A$ is an algebra over a base field $K$ then $\extcent{A}$ is a field extension of $K$ and $\extcent{A}\cap A = Z(A)$.
\end{rem}
There is an entirely internal characterization of $\extcent{A}$ which bears mentioning: the extended centroid of $A$ consists precisely of the equivalence classes $[(I, f)]$ where $f\colon I \ra A$ is an $(A,A)$-bimodule homomorphism:
\[
\extcent{A} = \setmid{(I,f)}{(0)\neq I\ideal A, f\in {\rm Hom}_{A}(_{A}I_{A}, _{A}A_{A})}/\sim,
\]
where $\sim$ is the equivalence described above.  (See \cite[Theorem 2.3.2]{BMM}.)

We note that this shows that the extended centroid we obtain from using the right Martindale ring of quotients is the same as that of the left Martindale ring of quotients.

  This observation allows one to extend the notion of rationality to prime ideals in non-noetherian rings (see Lorenz \cite{Lorenz1}, \cite{Lorenz2} for a discussion).
\begin{defn} {\em Let $K$ be a field and let $A$ be a $K$-algebra.  We say that a prime ideal $P$ of $A$
is \emph{rational} if $\extcent{A/P}$ is an algebraic extension of $K$.}
\end{defn}

\section{Stratification}
\label{Stratification}
In this section, we look at recent work of Aranda Pino, Pardo, and Siles Molina \cite{PPSM} on the prime spectrum of Leavitt path algebras and explain how it relates to the stratification theory of Goodearl and Letzter \cite{GL}.

The authors of \cite{PPSM} gave a remarkable correspondence between the prime spectrum of a Leavitt path algebra and a relatively simple set built from the underlying graph and a Laurent polynomial ring.  To construct this set, we quickly recall a few basic definitions.
\begin{defn}{\em Let $E=(E^0,E^1)$ be a directed graph.  We say that a nonempty subset $M\subseteq E^0$ is a \emph{maximal tail} if it satisfies the following properties:
\begin{enumerate}
\item If $v\in E^0$, $w\in M$ and $v\ge w$ then $v\in M$;
\item if $v\in M$ and $s^{-1}(v )\neq \emptyset$, then there is
$e\in E^1$ such that $s(e)=v$ and $r(e)\in M$;
\item if $v,w\in M$ then there is $y\in M$ such that $v\ge y$ and $w\ge y$.
\end{enumerate}
We let $\mathcal{M}(E)$ denote the set of maximal tails of $E$.}
\end{defn}

See Example \ref{tailsexample} below for some specific computations of maximal tails.



\begin{defn}
{\em We let $\mathcal{M}_{\gamma}(E)$ denote the subset of $\mathcal{M}(E)$ consisting of those maximal tails $M$ for which every cycle in $M$ has an exit in $M$.  We  denote by $\mathcal{M}_{\tau}(E)$ the set $ \mathcal{M}(E) \setminus \mathcal{M}_{\gamma}(E)$.   Rephrased, $M\in \mathcal{M}(E)$ has $M \in \mathcal{M}_{\gamma}(E)$ precisely when the graph $M$ satisfies Condition (L).}
\end{defn}
We note that by Property (3) of maximal tails, if $M \in \mathcal{M}_\tau$ then there is a {\it unique} (up to cyclic permutation) cycle $c$ in $M$ which has no exit.

Consistent with the remarks made at the end of the Section \ref{LPA}, if $M$ is a maximal tail of $E^0$, we will let (as an abuse of notation) $L_K(M)$ denote the Leavitt path algebra of the quotient graph $E/(E^0\setminus M)$.  Specifically,

\begin{rem}
Let $K$ be a field and let $E$ be a row-finite directed graph.  If $P$ is a graded prime ideal of $L_K(E)$, then we can identify the corresponding maximal tail $M:=E^0\setminus (E^0\cap P)$ of $P$ with the quotient graph $E/(E^0\cap P)$, and thus $L_K(E)/P\cong L_K(M)$.
\label{rem: 1}
\end{rem}

 We now describe the correspondence given by Aranda Pino, Pardo, and Siles Molina.  We note that $L_K(E)$ is a $\mathbb{Z}$-graded $K$-algebra, in which elements of $E^0, E^1$, and $(E^1)^*$ are given weights $0$, $1$, and $-1$, respectively.  Thus we have a distinguished subset of the prime spectrum of a Leavitt path algebra, consisting of the graded prime ideals. (Because $\Z$ is an ordered group, the set ${\rm gr}\mbox{-}{\rm Spec}(L_K(E))$ of  prime ideals of $L_K(E)$ which are graded equals the set of those graded ideals which are prime with respect to graded ideals; see e.g. \cite[Proposition II.1.4]{NO}.)  There is a bijective correspondence \cite[\S 3]{PPSM} between the set of graded prime ideals of $L_K(E)$ and maximal tails of $E$, given by
\begin{equation}
P\in {\rm gr}\mbox{-}{\rm Spec}(L_K(E)) \mapsto E^0\setminus \left( P\cap E^0\right).
\label{eq: 1}
\end{equation}
The inverse of this map is given by
\begin{equation}
M\in \mathcal{M}(E) \mapsto \sum_{v \in E^0 \setminus M} L_K(E)vL_K(E).
\end{equation}
We define
\begin{equation} {\rm Spec}_{\tau}(L_K(E)):={\rm Spec}(L_K(E))\setminus {\rm gr}\mbox{-}{\rm Spec}(L_K(E));
\end{equation}
that is, ${\rm Spec}_{\tau}(L_K(E))$ is the set of non-graded prime ideals of $L_K(E)$. Then Aranda Pino et al. \cite{PPSM} show there is an explicit bijective correspondence
\begin{equation}\label{eq: 2}
{\rm Spec}_{\tau}(L_K(E)) \to \mathcal{M}_{\tau}(E)\times  {\rm M}\mbox{-}{\rm Spec}(K[x,x^{-1}]),
\end{equation}
where M-Spec$(R)$ denotes the collection of maximal ideals of a ring $R$.

To describe this correspondence, we note that if $P$ is a non-graded prime ideal of $L_K(E)$, then by \cite[Lemma 2.7]{PPSM}  $M:=E^0\setminus (P\cap E^0)$ is in $\mathcal{M}_{\tau}(E)$.  Furthermore, if $M\in
\mathcal{M}_{\tau}(E)$ then there is a unique cycle (up to cyclic permutation) $c=c(M)$ in $M$ that does not have an exit.  Let $v$ be the vertex in $E^0$ at which $c$ is based; i.e., $v=s(c)=r(c)$.  Given a polynomial $f(x)=\sum_i \alpha_i x^i \in K[x,x^{-1}]$, we define \begin{equation}
f(c):=\sum_i \alpha_i c^i,
\end{equation} where we take $c^0$ to be $v$ and we take $c^{-1}$ to be $c^*$.
If $P$ is ungraded then there is some vertex $v\not\in P\cap E^0$, a cycle $c$ based at $v$ that does not have an exit in $E^0 \setminus P\cap E^0$, and an irreducible (so, nonzero) polynomial $f_P(x)\in K[x,x^{-1}]$ such that $f_P(c)\in P$.  (We note that $f_P(x)$ is unique up to multiplication by units and $c$ is unique up to cyclic permutation.)  The correspondence in Equation (\ref{eq: 2}) is given by
\begin{equation}
P\in {\rm Spec}_{\tau}(L_K(E)) \mapsto \left( E^0\setminus \left( P\cap E^0\right), (f_P(x))\right).
\end{equation}
and
\begin{equation}
\begin{array}{rl}
 (M, (f(x))\,) \in & \mathcal{M}_{\tau}(E)\times  {\rm M}\mbox{-}{\rm Spec}(K[x,x^{-1}])  ~\\
 \smallskip
\mapsto &   \sum_{v \in E^0 \setminus M} L_K(E)vL_K(E) + L_K(E) f(c(M)) L_K(E).
\end{array}
\end{equation}

Here now is the key description of the primitive ideals of $L_K(E)$.

\begin{prop}\label{descriptionofprimitives}  Let $E$ be a row-finite graph, and $K$ any field.

(i)  Let $P \in {\rm gr}\mbox{-}{\rm Spec}(L_K(E))$, and let $M = E^0 \setminus P\cap E^0$.  Then $P$ is primitive if and only if $M \in \mathcal{M}_{\gamma}(E)$.

(ii) Every $P \in {\rm Spec}_{\tau}(L_K(E))$ is primitive.

\end{prop}

\begin{proof}
(i) \  Let $H = P\cap E^0$.  Since $P$ is graded we have $P = I(H)$.  But as noted above, $L_K(E)/P = L_K(E)/I(H) \cong L_K(E / H) = L_K(M)$.  If $M \in \mathcal{M}_{\gamma}(E)$ then $M$ has Condition (L), so that $L_K(M)$ (and hence $P$) is primitive by \cite[Theorem 4.6]{PPSM}.  On the other hand, if $M\notin \mathcal{M}_{\gamma}(E)$ then $P$ is not primitive by \cite[Lemma 4.1]{PPSM}.

(ii) \   By the description given above, we have
$$P = \sum_{v\in E^0 \setminus  M} L_K(E)vL_K(E) + L_K(E) f(c(M)) L_K(E)$$ for some irreducible $f(x) \in K[x,x^{-1}]$, where $c(M)$ is a cycle without exits  based at the vertex $w = s(c)$.    Let $u$ denote the nonzero idempotent $w + P$ in the prime ring $L_K(E) / P$, and let $\varphi: wL_K(E)w \rightarrow K[x,x^{-1}]$ denote the isomorphism described in Proposition \ref{structureofvL(E)v}.  Then the quotient map $\overline{\varphi^{-1}}: K[x,x^{-1}] \rightarrow  (wL_K(E)w + P)/P = u(L_K(E)/P)u $ is surjective.    But the description of $P$ yields that ${\rm Ker}(\overline{\varphi^{-1}})\supseteq (f(x))$ for the irreducible polynomial $f(x)$, and since $\overline{\varphi^{-1}}$ is not the zero map, we get ${\rm Ker}(\overline{\varphi^{-1}}) = (f(x))$.  Thus  the nonzero corner $u (L_K(E)/P) u$ of the prime ring $L_K(E)/P$ is isomorphic to $K[x,x^{-1}]/(f(x))$, and thus is a field, and so in particular is primitive. We now apply \cite[Theorem 1]{LRS} to conclude that $P$ is a primitive ideal of $L_K(E)$.
\end{proof}

In particular, under the bijection given in Equation (\ref{eq: 1}), the set of graded primitive ideals is mapped bijectively to the set of those maximal tails $M$ having the property  that every cycle in $M$ has an exit.

 \begin{rem}\label{remarkdescriptionofprimitives}  We note that
 Proposition \ref{descriptionofprimitives} yields that  if $P \in {\rm Spec}(L_K(E))$ is not primitive, then in fact $P$ is graded.   Furthermore, we see that every Leavitt path algebra is Jacobson; that is, $J(L(E)/P) = \{0\}$ for every $P \in {\rm Spec}(L(E))$, as follows.   If $P$ is primitive then the result is immediate.  If $P$ is not primitive then since $P$ is graded we have $L(E)/P \cong L(M)$ where $M = E / H$ (see Remark \ref{rem: 1}).  But the Jacobson radical of any Leavitt path algebra is zero (e.g.,  \cite[Proposition 6.3]{AA3}).
 \end{rem}

We can recast the bijective correspondence given in Equations (\ref{eq: 1}) and (\ref{eq: 2}) in terms of the language of stratification of Goodearl and Letzter \cite{GL}.
For each prime ideal $P\in {\rm Spec}(L_K(E))$, we let $P_0$ denote the largest graded ideal that is contained in $P$.  By the correspondence given above, we see that $M:=E^0\setminus (P_0\cap E^0)$ is a maximal tail and hence $P_0$ is a graded prime ideal.  Given a maximal tail $M$, we define
\begin{equation}
{\rm Spec}_M(L_K(E)) = \{P\in {\rm Spec}(L_K(E))~:~P\cap E^0=E^0\setminus M\}.
\end{equation}
Given a maximal tail $M$, we call ${\rm Spec}_M(L_K(E))$ the \emph{stratum corresponding to} $M$.
The graded prime ideals are in bijective correspondence with the maximal tails via the correspondence in Equation (\ref{eq: 1}), and thus
\begin{equation}
 {\rm Spec}(L_K(E)) = \bigsqcup_{M\in \mathcal{M}(E)} {\rm Spec}_M(L_K(E)).
 \end{equation}
 Moreover, if $M\in \mathcal{M}_{\gamma} = \mathcal{M}(E)\setminus \mathcal{M}_{\tau}(E)$ then $ {\rm Spec}_M(L_K(E)) $ is a singleton, consisting of exactly the graded prime ideal $I(E^0\setminus M)$ corresponding to $M$.  If $M\in \mathcal{M}_{\tau}(E)$ then we see that $ {\rm Spec}_M(L_K(E))$ is homeomorphic to ${\rm Spec}(K[x,x^{-1}])$.

 This discussion, together with Proposition \ref{descriptionofprimitives}, thus yields
  the following description of the prime spectrum of a Leavitt path algebra.

\begin{thm}\label{strat} (Stratification for Leavitt path algebras) Let $K$ be a field and let $E$ be a row-finite directed graph.  Then
\begin{equation} {\rm Spec}(L_K(E)) = \bigsqcup_{M\in \mathcal{M}(E)} {\rm Spec}_M(L_K(E)).
\end{equation}
Moreover,
\[ {\rm Spec}_M(L_K(E))\cong \left\{ \begin{array}{ll} {\rm Spec}(K[x,x^{-1}]) & {\rm if ~} M\in \mathcal{M}_{\tau}(E),\\
{\rm Spec}(K) & {\rm otherwise.} \end{array} \right. \]
In particular if the set of maximal tails is finite, then the prime spectrum is the union of a finite number of strata, with each stratum homeomorphic to a $0$- or $1$-dimensional torus, and the primitive ideals are precisely those that are maximal in their stratum.
\end{thm}

The following consequence of Proposition \ref{descriptionofprimitives} and Theorem \ref{strat} will be useful later.

\begin{cor}\label{containsverticescorollary}
 Let $P \in {\rm Spec}(L_K(E))$.  Then $P$ is primitive if and only if for every $Q \in {\rm Spec}(L_K(E))$ having $P\subsetneq Q$,  there exists a vertex $v$ with $v\in Q \setminus P$.
 \end{cor}

 \begin{proof}  Let $M = E^0 \setminus (E^0\cap P)$ as usual.  Suppose $P$ is primitive.   If $P$ has $M \in \mathcal{M}_{\gamma}$ then the result is clear by the description of ${\rm Spec}(L_K(E))$.  In the other case, suppose  $P = I(H) + L(E)f(c)L(E)$ (i.e., suppose $P$ is not graded).  In particular,  $P\cap E^0 = H$.  Suppose $P \subsetneq Q$ with $Q\in {\rm Spec}(L_K(E))$.  Then since $P$ is maximal in its stratum, $Q$ cannot be of the form $I(H) + L(E)g(c)L(E)$ for this particular $H$.   If $Q$ is graded,  $Q = I(H')$ for some hereditary saturated subset $H'$ of $E$.  But then by the bijection between hereditary saturated subsets and graded ideals we have $H \subseteq H'$, and equality is not possible since $P \neq I(H)$.  If on the other hand $Q$ is not graded, then $Q = I(H') + L(E)g(c')L(E)$, and again necessarily $H \subsetneq H'$.

 Conversely, if $P$ is not primitive then any ideal maximal in the stratum of $P$ contains the same vertices as does $P$.
 \end{proof}

Although we will not specifically need the following ideas in order to establish our main result (Theorem \ref{thm: main}), to provide context for the results presented herein it is worthwhile to note that Theorem \ref{strat} can be recast in terms of rational torus actions in case $K$ is infinite.
Let $H$ be an affine algebraic group, let $K$ be an infinite field, and let $A$ be an associative $K$-algebra with a rational action of $H$ by automorphisms.   We recall that the action
of $H$ on $A$ is called \emph{rational} if there is a $K$-algebra homomorphism
$$\Delta : A \to A \otimes K[H],\qquad a\mapsto \sum_i b_i \otimes f_i$$
such that
$$h\cdot a = \sum b_i \otimes f_i(h)$$ for $h\in H$. In this case, $K[H]$ denotes the (Hopf) algebra of regular functions on $H$.  In the case that $A$ is a finitely generated commutative reduced $K$-algebra, then saying that the action of $H$ on $A$ is rational is equivalent to saying that the corresponding map $H\times {\rm Spec}(A)\to {\rm Spec}(A)$ is an algebraic morphism.

We endow $L_K(E)$ with a rational $K^{\times}$ action by declaring that if $h\in K^{\times}$ and $a\in L_K(E)$ is homogeneous of degree $n$ then
$h\cdot a = h^n a$.  (For a complete description of this action, see \cite[Section 1]{AALP}.)  We note that $K[H]\cong K[x,x^{-1}]$ and we have a $K$-algebra homomorphism
$$\Delta : L_K(E)\to L_K(E)\otimes_K K[x,x^{-1}],$$ defined by
$$\Delta(a) = a\otimes x^n$$ for $a$ homogeneous of degree $n$.  This shows that the action is indeed rational.  It is not hard to show that the $H$-invariant ideals are precisely the graded ideals when $K$ is infinite, see e.g. \cite[Proposition 1.6]{AALP} (where ``$H$-invariance" is called ``gauge invariance").  As is thus appropriately customary, we let $H$-${\rm Spec}(L_K(E))$ denote the set of $H$-invariant prime ideals of $L_K(E)$, and, given an $H$-prime $P$, we let ${\rm Spec}_P(L_K(E))$ denote the set of primes $Q$  in $L_K(E)$ with the property that $P$ is the largest $H$-prime contained in $Q$.  Given an $H$-prime $P$ we call ${\rm Spec}_P(L_K(E))$ the $H$-\emph{stratum of} $P$.   $H$-primes are therefore exactly the set of graded prime ideals in the case that $K$ is infinite, and are in bijective correspondence with the maximal tails.  Furthermore, if $M$ is the maximal tail corresponding to $P$ then ${\rm Spec}_P(L_K(E))={\rm Spec}_M(L_K(E))$.  Thus the stratification given in Theorem \ref{strat} can be recast in terms of group actions, with:  the $H$-primes in bijective correspondence with the maximal tails of $E$;  the $H$-strata being exactly the strata;  primitive $H$-primes being exactly those that are maximal in their $H$-stratum;  and the non-primitive $H$-primes being those with a $1$-dimensional $H$-stratum.

One should compare this characterization of the prime spectrum of $L_K(E)$ with Goodearl and Letzter's stratification theory \cite{GL} (see also Brown and Goodearl \cite[Chapter II.2]{BG}) for various quantum groups.   In this light, the following example is of note.

 \begin{exam}\label{tailsexample}
{\rm Let $E$ denote the directed graph given in Figure \ref{tailsfigure}.  It is straightforward to see that there are four hereditary saturated subsets whose complements in $E^0$ are maximal tails:  $H_1 = \emptyset$, $H_2 = \{v_1,v_2\}$,  $H_3 = \{v_2,v_3\}$, and $H_4 = \{v_1,v_2,v_3\}$.  (We note that the set $\{v_2\}$ is hereditary and saturated, but its complement is not a maximal tail; thus $(v_2)$ is a graded ideal which is not prime.)   Furthermore, the corresponding maximal tails  have $M_1, M_4 \in  \mathcal{M}_{\gamma}$, while $M_2,M_3 \in \mathcal{M}_{\tau}$.   Thus, using Theorem \ref{strat}, we see that the strata of ${\rm Spec}(L_{\mathbb{C}}(E))$ can be described as in the diagram.
}
\end{exam}

In particular, notice that if $E$ is the directed graph given in Figure \ref{tailsfigure}, then the prime spectrum of $L_K(E)$ is homeomorphic to the prime spectrum of the quantum plane, $\mathbb{C}_q\{x,y\}/(xy-qyx)$ with $q\in \mathbb{C}$ not a root of unity (see Figure $5$).

  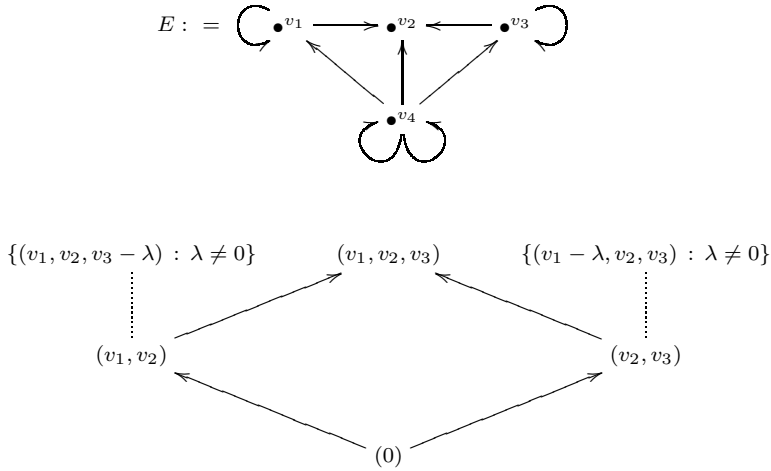
\begin{figure}\label{tailsfigure}
$$E: \ \ = \ \ \ \ \  \xymatrix{{\bullet}^{v_1}\ar@(ul,dl)  \ar[r]
& \bullet^{v_2}  & \bullet^{v_3} \ar@(ur,dr)  \ar[l] \\
& \bullet^{v_4}\ar[ul] \ar[u] \ar[ur] \ar@(d,l) \ar@(d,r) & &
}$$
\\
\bigskip
\bigskip
$$
\xymatrix{\{ (v_1,v_2,v_3-\lambda)\,:\, \lambda\neq 0\} \ar@{.}[d] &   (v_1,v_2,v_3)
&  \{(v_1-\lambda ,v_2,v_3)\,:\, \lambda\neq 0\} \ar@{.}[d]\\
 (v_1,v_2) \ar[ur] &  & (v_2,v_3)\ar[ul] \\   & (0)\ar[ul] \ar[ur]  & } $$
\caption{(above) \ A directed graph $E$ and the stratification of the prime spectrum of $L_{\mathbb{C}}(E)$}
\end{figure}

\bigskip
\bigskip

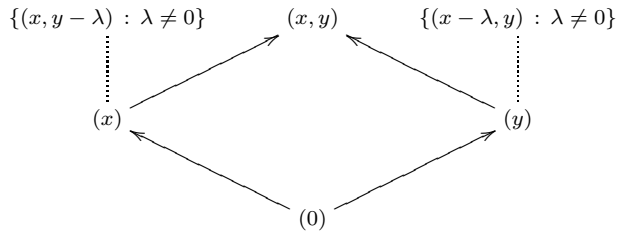
\begin{figure}\label{quantumplanefigure}
$$
\xymatrix{\{ (x,y-\lambda)\,:\, \lambda\neq 0\} \ar@{.}[d] &   (x,y)   &  \{(x-\lambda,y)\,:\, \lambda\neq 0\} \ar@{.}[d]\\
 (x)\ar[ur] &  & (y)\ar[ul] \\   & (0)\ar[ul]\ar[ur]  & } $$
\caption{(above) \ Stratification of the prime spectrum of the quantum plane}
\end{figure}
 Stratification is especially interesting in the case that the number of $H$-primes is finite.  Note that if $E$ is a finite directed graph then $\mathcal{M}(E)$ is finite and so $H$-${\rm Spec}(L_K(E))$ is finite in this case.


\section{The Dixmier-Moeglin Equivalence}
\label{DM}

The following well-known property of primitive ideals (see e.g. Brown and Goodearl \cite[Lemma II.7.13]{BG}) will play a key role in our main result.

\begin{prop}\label{algebraicbydimensionprop}  Let $K$ be a field and let $A$ be a primitive $K$ algebra.  If ${\rm dim}_K(A)<\# K$, then $\extcent{A}$ is algebraic over $K$.\label{rem: 2}
\end{prop}
\begin{proof} Let $V$ be a faithful simple right $A$-module  We note that $\extcent{A}$ embeds in ${\rm End}_A(V)$ as follows.  Given a nonzero ideal $I$ and a bimodule homomorphism $f:I\to A$, we have an element $[(I,f)]\in \extcent{A}$.  We may pick an equivalence class representative in which $I=(a)$ for some nonzero $a\in A$.  Let $b=f(a)$.  Then $VaA=V$ and we create an endomorphism $\phi_f:V\to V$ by $vax\mapsto vbx$ for $v\in V$ and $x\in A$.  Thus $${\rm dim}_K(\extcent{A})\le {\rm dim}\left( {\rm End}_A(V)\right) \le {\rm dim}_K\, A < \# K.$$
But this inequality yields  that $\extcent{A}$ is algebraic over $K$ (see e.g. \cite[proof of Proposition II.7.16]{BG}).
\end{proof}

We are now in position to prove our main result.

\medskip

{\bf Theorem \ref{thm: main}}.
Let $E$ be a finite directed graph and let $K$ be a field.  Then the Dixmier-Moeglin equivalence holds for prime ideals of $L_K(E)$; that is, these conditions on a prime ideal $P$ of $L_K(E)$ are equivalent: \begin{enumerate}
\item $P$ is primitive;
\item $P$ is rational;
\item $P$ is locally closed in ${\rm Spec}(L_K(E))$.
\end{enumerate}

\begin{proof}
   Let $P$ be a prime ideal in $L_K(E)$. As usual we let $M$ denote $E^0 \setminus (E^0 \cap P)$.

\smallskip

We start by showing that $P$ is primitive if and only if $P$ is locally closed.
Suppose $P$ is primitive.  Using Corollary \ref{containsverticescorollary} we get
  $$\bigcap\{Q \ | \ Q\in {\rm Spec}(L(E)), P \subsetneq Q\} \ \supseteq \ \bigcap\{P + L(E)vL(E) \ | \ v\not\in P\}.$$
  But this latter set contains  the (finite) product of ideals  $$\prod_{v\not\in P}(P + L(E)vL(E)).$$  Since $P$ is prime and none of the $P + L(E)vL(E)$ is contained in $P$ this product must contain $P$ strictly, which yields that $P$ is locally closed.

\medskip

Now suppose $P$ is not primitive.  Since $L_K(E)$ is Jacobson (Remark \ref{remarkdescriptionofprimitives}), it is well known (e.g., see Brown and Goodearl \cite[Lemma II.7.11]{BG}) that $P$ is not locally closed; we include a short proof here for clarity.  We show $$\bigcap\{Q \ | \ Q\in {\rm Spec}(L(E)), P \subsetneq Q\} \ = \ P.$$   Since $J(L(E)/P) = \{0\}$  (Remark   \ref{remarkdescriptionofprimitives}), we have
 $$\bigcap\{\mathcal{Q} \ | \ \mathcal{Q}\in {\rm Prim}(L(E)/P)\} =\{0\}.$$
Since the primitive ideals of $L(E)/P$ are precisely the images of the primitive ideals of $L(E)$ which contain $P$, this gives
 $$ \bigcap\{Q \ | \ Q\in {\rm Prim}(L(E)), P\subseteq Q\} \ = \ P.$$
But $P$ itself is not primitive, so this last equality can be written as
 $$ \bigcap\{Q \ | \ Q\in {\rm Prim}(L(E)), P\subsetneq Q\} \ = \ P  ,$$
which immediately gives
$$ \bigcap\{Q \ | \ Q\in {\rm Spec}(L(E)), P\subsetneq Q\} \ = \ P$$
as desired.

\smallskip

We now show that $P$ is primitive if and only if $P$ is rational.   First, suppose $P$ is not primitive.  Then by Remark \ref{rem: 1} $L_K(E)/P \cong L_K(M)$, where $M$ is a graph which contains a cycle without exits.  Denote the cycle by $c$, and its base vertex by $v=s(c)=r(c)$.  In this setting we build an element of the extended centroid of $L_K(M)$ which is transcendental over $K$; we start by defining the key  bimodule homomorphism.  Specifically,
consider the map
$$f : L_K(M) v L_K(M) \rightarrow L_K(M) \ \ \mbox{defined by setting} \ \  f(avb)=acb$$
 for all $a,b\in L_K(M)$,  and extending $K$-linearly.  Once the well-definedness of $f$ is established, $f$ will clearly be  an $L_K(M)$-$L_K(M)$ bimodule homomorphism.

We first consider the case when $K$ is infinite.  It suffices to show that whenever
$\sum_{i=1}^m a_i v b_i = 0$, then $\sum_{i=1}^m a_i c b_i =0$.  To the contrary, suppose there are elements $a_1,\ldots,a_m,b_1,\ldots,b_m$ in $L_K(M)$ such that
$\sum_{i=1}^m a_i v b_i =0$, while   $y:=\sum_{i=1}^m a_i c b_i$ is nonzero.
Then we get
$$y=\sum_{i=1}^m a_i(c-\lambda v)b_i$$
 for all $\lambda\in  K$.
In particular, $y$ is in the ideal $L_K(M) (c-\lambda v)L_K(M)$ for all $\lambda \in K$, and so
$$\bigcap_{\lambda \in K} L_K(M) (c-\lambda v)L_K(M)$$
 is a nonzero ideal of $L_K(M)$.
But since $v$ is idempotent we have
$$ v\cdot  [ \bigcap_{\lambda \in K}  L_K(M) (c-\lambda v)L_K(M) ]\cdot v \ = \ \bigcap_{\lambda \in K} v L_K(M) (c-\lambda v)L_K(M)v $$
$$= \bigcap_{\lambda \in K} vL_K(M)v (c-\lambda v) vL_K(M)v,$$
so that $\cap_{\lambda \in K} vL_K(M)v (c-\lambda v) vL_K(M)v$ is nonzero (as it contains $y$).

  Now using the isomorphism $\varphi$ described in Proposition \ref{structureofvL(E)v}, this yields that
 $$0 \ \neq \ \bigcap_{\lambda \in K} K[x,x^{-1}](x-\lambda)K[x,x^{-1}].$$
 But, as $K$ is infinite, this last intersection is $0$ in $K[x,x^{-1}]$, a contradiction.

Thus $f$   is well-defined when $K$ is infinite.  To establish the well-definedness of $f$  for all fields,  suppose $K$ is finite,  and let $K'$ be an infinite extension of $K$.  Then
$L_{K'}(E) = L_K(E) \otimes_K K'$, and, by the just-established result,   the  map $f^{\prime}:L_{K'}(E)vL_{K'}(E) \rightarrow L_{K'}(E)$ given by $f^{\prime}: avb \mapsto acb$ is well-defined. But  the restriction of $f^{\prime}$  to $L_K(E)vL_K(E)$ is $f$, and so  we get that $f$ is well-defined as well.

Since $f$ is a bimodule homomorphism, by the remarks made subsequent to Definition \ref{extendedcentroiddef} we have that the equivalence class of  $(f,L_K(M)vL_K(M))$ is in $Z(\mart{L_K(M)}$.
We now show that this equivalence class is transcendental over $K$.

If not, there is some nonzero ideal $I$ of $L_K(M)$ such that $f, f^2, \ldots, f^d$ are defined on $I$, and elements $a_0, a_1, \dots, a_d$ of $K$ (not all zero) for which the restriction of
$$g:=a_0  + a_1 f + \cdots +a_d f^d$$
 on $I$ is identically zero.
But $L_K(M)$ is prime and hence $$0 \ \neq \ L_K(M)vL_K(M)\cdot I \cdot L_K(M)vL_K(M) = L_K(M)vIvL_K(M),$$ so that $0\neq vIv = I\cap vL_K(M)v$.
In particular,  $g$  is zero on the nonzero ideal $I \cap v L_K(M)v$ of $v L_K(M)v$.
So again invoking the isomorphism of Proposition \ref{structureofvL(E)v}, this yields that there is a nonzero ideal $T$ of $K[x,x^{-1}]$ on which $\varphi \circ g \circ \varphi^{-1}$ is zero. Now pick any nonzero $z = p(x) \in T$.   Note that by definition we have  $f(v) = c$,  $f(c) = f(cvv) = ccv = c^2$, and,   more generally,  for any $1\leq j\leq d$ we have $f^j(c) = c^{j+1}$.  Using this,   we get easily  that
$$\varphi \circ g \circ \varphi^{-1}(p(x)) =  (a_0+a_1 x+\cdots+a_d x^d ) p(x),$$
 which is nonzero in $K[x,x^{-1}]$, a contradiction.

 Thus the equivalence class of $(f, L_K(M)vL_K(M))$ is transcendental over  $K$, thereby yielding  that $Z(\mart{L_K(M)} \cong Z(\mart{L_K(E)/P} $ is not algebraic over $K$, so that $P$ is not rational.
 \bigskip


Thus it only remains to show that if $P$ is primitive then $P$ is rational.
  So suppose that $P$ is primitive.  Let $K'$ be an uncountable purely transcendental field extension of $K$.  Note that
$L_{K'}(E)\cong L_K(E)\otimes_K K'$ and
$P':=P\otimes_K K'$ is a prime ideal of $L_{K'}(E)$ by Theorem \ref{strat}.  Moreover, $P'$ is maximal in its stratum, as $P$ is maximal in its stratum (and $K'$ transcendental over $K$ ensures that irreducibility of polynomials passes between $K[x,x^{-1}]$ and $K'[x,x^{-1}]$).  Thus  $P'$ is primitive by Theorem \ref{strat}.
We let $R=L_K(E)/P$ and let $R'=L_{K'}(E)/P'$.
Notice we have an injective homomorphism $$\Phi:\extcent{R}\otimes_K K'\to \extcent{R'}$$ defined as follows.  If $I$ is a nonzero ideal of $R$ and $f:I\to A$ is a bimodule homomorphism then if $[(I,f)]\in \extcent{R}$ is the equivalence class associated to $f$, we define $\Phi([(I,f)])$ to be
the class $[(I',f')]$, where $f'=f\otimes {\rm id}: I\otimes_K K' \to R\otimes_K K' = R'$.  Since $R'$ is primitive and
$${\rm dim}_{K'} R' \ < \ \# K'
$$
 (note that as $E$ is finite we have ${\rm dim}_{k} (L_k(E))$ is at most countable for any field $k$), we conclude that $\extcent{R'}$ is algebraic over $K'$ by Proposition \ref{algebraicbydimensionprop}.
Given $[(I,f)]\in\extcent{R}$ we see that the bimodule homomorphism $[(I',f')]$ is algebraic over $K'$.  Thus there exists some nonzero ideal $J'\subseteq I'$ of $R'$ and some natural number $n$ such that
$$K' (f\otimes {\rm id})^n|_{J'} + \cdots + K'(f\otimes {\rm id})|_{J'}+K'{\rm id}|_{J'}$$ is not direct.   Thus
$$K f^n|_J + \cdots + K f|_J + K{\rm id}|_J$$ is not direct, where $J=J'\cap R$.  Hence $[(I,f)]$ is algebraic over $K$ and so we see that $\extcent{R}$ is algebraic over $K$ and thus $P$ is rational.  This completes the proof that $P$ primitive implies $P$ rational, and thereby completes the proof of the main theorem.
 \end{proof}

\bigskip

We conclude this article with an easy construction which yields nontrivial examples of prime spectra for Leavitt path algebras.        For any positive integer $n$ let $S \sqcup T$ be any partition of $\{1,2,\ldots,n\}$.   Let $E = E_{n,S}$ be the graph defined as follows:
  $$E^0 = \{v_1, v_2, \ldots, v_n\}$$
  $$E^1 = \{e_1, e_2, \ldots e_{n-1}\} \ \sqcup \ \{f_1, f_2, \ldots ,f_n\} \ \sqcup \ \{g_{i_1}, \ldots, g_{i_{|S|}}\},$$
  where
  $$r(e_i) = v_i, \ s(e_i) = v_{i+1}  \ (1\leq i \leq n-1); \ \ \ s(f_i)=r(f_i)=v_i \ (1\leq i \leq n); $$
   $$\ \mbox{and} \ \ \ s(g_{i_j}) = v_{i_j} \ (1 \leq i_j \leq |S|).$$

  In words, $E$ is the graph with $n$ vertices, connected one to the next by the edges $e_1, \ldots ,e_{n-1}$, for which there are two loops at  $v_i$ when $i\in S$, and one loop at $v_i$ otherwise.

  We define $H_0 = \emptyset$, and for each $1\leq i \leq n$ we define $H_i = \{v_1,v_2,\ldots,v_i\}$.  Then a straightforward check yields that the $H_i$ are the only  hereditary saturated subsets of $E^0$, and that they form a chain
  $$\emptyset = H_0 \subsetneq H_1 \subsetneq \cdots \subsetneq H_n = E^0.$$
  Thus  the collection of graded ideals in $L_K(E)$ is the chain
  $$ \{0\} = I(H_0) \subsetneq I(H_1) \subsetneq \cdots \subsetneq I(H_n) = L_K(E).$$
  Furthermore, for $0\leq i \leq n-1$,  each $M_i = E^0 \setminus H_i$ is easily seen to be a maximal tail, so that in fact all the (proper) graded ideals of $L_K(E)$ are graded prime.  Finally,  the cycle (loop) $f_{i+1}$ based at vertex $v_{i+1}$ of $M_i$ has an exit in $M_i$ precisely when there is a second loop $g_{i+1}$ based at $v_{i+1}$.  Thus  we see that ${\rm gr}\mbox{-}{\rm Spec}(L_K(E))$ is the chain $\{P_i = I(H_i) \ | \ 0\leq i \leq n-1\}$ of length $n$, where, for each $P_i$ in the chain, $P_i$ is primitive if and only if $i\in S$.

    For instance, if $n=4$ and $S = \{1,3\}$, then

$$E_{4,\{1,3\}} \  = \ 
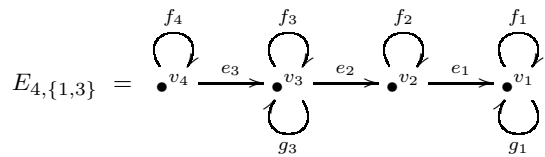
\xymatrix{{\bullet}^{v_4}\ar@(ul,ur)^{f_4}  \ar[r]^{e_3}
& \bullet^{v_3}\ar@(ul,ur)^{f_3}\ar@(dr,dl)^{g_3} \ar[r]^{e_2} & \bullet^{v_2}\ar@(ul,ur)^{f_2}  \ar[r]^{e_1} & \bullet^{v_1}\ar@(ul,ur)^{f_1}\ar@(dr,dl)^{g_1}}$$

\bigskip

In a subsequent article, we consider these naturally arising questions:
\begin{enumerate}
\item What is the structure of the poset of graded prime ideals of a Leavitt path algebra under inclusion?
\item Can we enumerate the prime and primitive graded prime ideals in nice families of Leavitt path algebras?
\end{enumerate}
These questions are in fact quite natural to ask, as algebras which support a rational algebraic group action that stratifies the prime spectrum into finitely many strata often have the property that the prime spectrum has an interesting combinatorial structure.  For example, in the case that we are dealing with the quantized coordinate ring of $m\times n$ matrices, Cauchon \cite{C} has given a beautiful closed formula for the number of torus-invariant prime ideals.  Launois \cite{Launois} has shown that the poset of torus-invariant prime ideals is isomorphic to the poset consisting of permutations $\sigma\in S_{n+m}$ such that $i-m\le \sigma(i)\le i+n$ for all $i\in \{1,2,\ldots ,n+m\}$, in which the order is the reverse Bruhat order on permutations.  Recently, the primitive torus-invariant prime ideals have been enumerated \cite{BCL}, \cite{BLL}, \cite{BLN}.

\end{document}